\documentclass[a4paper,10pt]{amsart}
\usepackage{amssymb,amsmath,amsfonts,amsthm}
\usepackage[dvips]{graphicx}
\usepackage{psfrag}
\usepackage{xcolor}

\theoremstyle{plain}
\newtheorem{main}{Theorem}

\newtheorem{theorem}{Theorem}[section]
\newtheorem{lemma}[theorem]{Lemma}
\newtheorem{proposition}[theorem]{Proposition}
\newtheorem{corollary}[theorem]{Corollary}
\theoremstyle{remark}
\newtheorem{remark}[theorem]{Remark}
\newtheorem{definition}[theorem]{Definition}

\newtheorem{Claim}[theorem]{Claim}

\newtheorem{conjecture}[theorem]{Conjecture}

\newcommand{\Gibb}{\operatorname{Gibb}}
\newcommand{\s} {\sigma}

           \def\ea{\end{array}}
          \def\ec{\end{center}}
     \def\ed{\end{description}}
        \def\ee{\end{equation}}
       \def\eea{\end{eqnarray}}
     \def\eeaa{\end{eqnarray*}}
 \def\et{\end{thebibliography}}

\def\Diff{{\rm Diff}}

\def\SPH{\operatorname{SPH}}

\def\cA{{\mathcal A}}
\def\cC{{\mathcal C}}

\def\cB{{\mathcal B}}

\def\cF{{\mathcal F}}

\def\cW{{\mathcal W}}

\def\loc{\operatorname{loc}}
\def\diam{\operatorname{diam}}

\def\vep{\varepsilon}

\def\TT{{\mathbb T}}

\title{Invariance principle and rigidity of high entropy measures}
\author{Ali Tahzibi and Jiagang Yang}
\date{\today}

\address{Departamento de Matem\'atica,
  ICMC-USP S\~{a}o Carlos-SP, Brazil.}
\email{tahzibi@icmc.usp.br}

\address{Departamento de Geometria, Instituto de Matem\'atica e Estat\'istica, Universidade
Federal Fluminense, Niter\'oi, Brazil}
\email{yangjg\@@impa.br}

\begin{document}

\begin{abstract}

 A deep analysis of the Lyapunov exponents, for stationary sequence of matrices  going back to Furstenberg \cite{F}, for more general linear cocycles by Ledrappier \cite{L} and  generalized to the context of non-linear cocycles by Avila and Viana \cite{AV}, gives an invariance principle for invariant measures with vanishing central exponents. In this paper, we give a new criterium formulated in terms of entropy for the invariance principle and in particular, obtain a simpler proof for some of the known invariance principle results.

As a byproduct, we study ergodic measures of partially hyperbolic diffeomorphisms
whose center foliation is 1-dimensional and forms a circle bundle. We show that for
any  such $C^2$ diffeomorphism which is accessible, weak hyperbolicity of ergodic measures of high entropy implies that the system itself is of rotation type.
\end{abstract}

\maketitle

\setcounter{tocdepth}{1} \tableofcontents

\section{Introduction}

The topological entropy and the metric entropy of dynamical systems are celebrated invariants in the moduli of topological and metric conjugacy. For the continuous dynamics of compact spaces, they are related by the variational principle: The topological entropy is the supremum over the metric entropies of all  probability invariant measures.

A classical problem in ergodic theory is to determine whether the supremum of metric entropies is attained by invariant measures of maximal entropy. The number of such maximizing measures is another interesting question. 


The Lyapunov exponents are also  important numbers which measure the complexity of dynamics. They are defined almost everywhere with respect to any invariant probability measure. By Oseledets' theorem, for an invariant ergodic measure of a diffeomorphism $f\in\Diff^1(M)$ over a manifold $M$ with dimension $d$, there are $d$ numbers $ \lambda_1 \leq \lambda_2 \cdots \leq \lambda_d$ such that for a.e. $x \in M$ and for any $v \in TM \setminus \{0\}$, we have $\lim_{n \rightarrow \infty} \frac{1}{n} \log \|Df_x^n (v)\| = \lambda_i$ for some $1 \leq i \leq d.$ The $d-$numbers $\lambda_i$ are the \emph{Lyapunov exponents} of $(f, \mu).$ A measure $\mu$ is called \emph{hyperbolic} if all its Lyapunov exponents are non-zero. 

The entropy and Lyapunov exponents of smooth diffeomorphisms are related by celebrated Ruelle's inequality and Pesin's formula: ``The entropy is smaller than the sum of the positive Lyapunov exponents and the equality is equivalent to the smoothness of measure along unstable manifolds".
For surface diffeomorphisms, using Ruelle's inequality it is not difficult to see that any measure of non-zero entropy is hyperbolic.

Partially hyperbolic dynamics constitutes a successful branch of dynamics beyond uniformly hyperbolic systems (See next section for definitions).
For partially hyperbolic diffeomorphisms, the non-hyperbolicity  of invariant measures  comes from vanishing exponent in the central bundle.
One interesting problem in smooth ergodic theory is to verify the abundance of partially hyperbolic dynamics with natural measures (for instance, volume or measures of maximal entropy) with non-vanishing central exponents.
An approach to this problem is to find mechanisms to remove zero central exponents  (see \cite{GT} for a survey). Another way is to study the possible rigid properties of the systems with non-hyperbolic natural measures.

A deep analysis of the Lyapunov exponents, for stationary sequence of matrices  going back to Furstenberg \cite{F}, more general linear cocycles by Ledrappier \cite{L} and  generalized to the context of non linear cocycles by Avila and Viana \cite{AV}, gives an invariance principle for invariant measures with vanishing central exponents.  Speaking in the spirit  of  the works of Ledrappier and Avila-Viana, vanishing exponents on the central direction  reveals the ``deterministic behavior of the invariant measure along the central foliation". See subsection (\ref{invariance1}) in Preliminaries  for a more precise interpretation. The invariance principle had been noticed also by Baxendale \cite{Bax} (see also a neat proof in the circle action case in the work of Deroin-Klepstyn and Navas \cite{DKN} and the result of Crauel \cite{cra}.) 

In this paper we work with the notion of entropy along expanding foliations \cite{Y} and give a simple criterium for the invariance principle. In particular, we obtain a simpler proof for the invariance principle which was  formulated in terms of the Lyapunov exponents in the previous known results. We emphasize that  in the proof of invariance principle by Ledrappier and Avila-Viana a notion of entropy (Kullback information, see \cite{L}) along central foliation is hidden.

As a byproduct we obtain a rigidity result for partially hyperbolic diffeomorphisms with one dimensional compact central foliation.
We consider high entropy ergodic measures of partially hyperbolic dynamics with one dimensional compact central leaves and prove ``strong hyperbolicity" of such measures for typical dynamics. By strong hyperbolicity we mean that all high entropy ergodic measures have center exponent uniformly bounded away from zero.  So, our result is a rigidity statement for partially hyperbolic diffeomorphisms with one dimensional center bundle (see the precise setting in what follows): if there are high entropy invariant measures which are weakly hyperbolic (with  central Lyapunov exponent converging to zero) then in fact the dynamics is conjugate to isometric extension of Anosov homeomorphism. In particular, our result sheds light on some conjectures in smooth ergodic theory, specially in the quest for non-hyperbolic ergodic measures  related to a conjecture of D\'{i}az-Gorodetski in \cite{DG} (see  section \ref{quest} for more details). We thank S. Crovisier for the comments on the relation of our work with the iterated function systems setting and discussions on the Ledrappier-Young results.

\section{Statement of results}

Throughout this paper we will work with
partially hyperbolic diffeomorphisms. $f: M \rightarrow M$ is \emph{partially hyperbolic} if
there is a $Tf$-invariant splitting of the tangent
bundle $TM = E^s\oplus E^c \oplus E^u$, such that for all unit vectors
$v^\s\in E^\s_x\setminus \{0\}$ ($\s= s, c, u$) with $x\in M$ we have:
$$\|T_xfv^s\| < \|T_xfv^c\| < \|T_xfv^u\| $$
for some suitable Riemannian metric.  Furthermore $f$ satisfies:
$\|Tf|_{E^s}\| < 1$ and $\|Tf^{-1}|_{E^u}\| < 1$.  For partially hyperbolic diffeomorphisms, it is a well-known fact that there are foliations $\cF^\s$ tangent to the distributions $E^\s$
for $\s=s,u$ . The leaf of $\cF^\s$
containing $x$ will be called $\cF^\s(x)$, for $\s=s,u$. \par In general the central bundle $E^c$ may not be tangent to an invariant foliation. However, whenever it exists we denote it by $\mathcal{F}^c.$

Our first main result is a criterium in terms of entropy for the so-called invariance principle (See (\ref{invariance1}) and (\ref{invariance2}).) for cocycles over {\it Anosov homeomorphisms} (See Preliminaries section). 
The building block of the proof uses arguments similar to Ledrappier \cite{L}, Ledrappier-Young \cite{LY}. Our proof depends on the analysis of
partial entropy along expanding foliation of the measures (see \cite{L1}, \cite{LY2} when this notion firstly was introduced  and \cite{Y} for a recent generalization).
In particular, our approach permits  us to give an interpretation of the proof of the invariance principle obtained by Avila-Viana \cite{AV} in the case of cocycles on fiber bundles with compact fibers and over Anosov homeomorphism, without using deformation of cocycles.

Let $f: M \rightarrow M$ be a partially hyperbolic dynamics satisfying the following conditions (See Preliminaries, section (\ref{preliminaries}), for the definitions.):
\begin{itemize}
\item H1. $f$ is dynamically coherent with all center leaves compact,
\item H2. $f$ admits global holonomies,
\item H3. $f_c$ is a transitive topological Anosov homeomorphism, where $f_c$ is the induced dynamics satisfying $f_c \circ \pi = \pi \circ f$ and $\pi : M \rightarrow M/\mathcal{F}^c$ is the natural projection to the space of central leaves.
\end{itemize}

A large class of partially hyperbolic dynamics denoted by fibered partially hyperbolic systems satisfy (H1) and (H2) and all known examples satisfy (H3). In particular, it is shown in \cite{HP} that, 
over any 3 dimensional Nilmanifold different from the torus, every partially hyperbolic diffeomorphism satisfies (H1), 
(H2) and (H3).

 Denote by $h (\mu, \mathcal{F}^u)$  the ``entropy along unstable foliation of $f$" (See sub section \ref{entropyfoliation} for the details).  For an $f-$invariant measure  $\mu,$ let $\nu = \pi_*(\mu)$ and $\{\mu^u_x\}, \{\nu^u_{\pi(x)}\}$ denote respectively conditional measures of $\mu$ and $\nu$ along suitable measurable partitions sub-ordinated to the unstable foliation of $f$ and $f_c.$ We say $\mu \in Gibb^u_{\nu}(f)$ if $\pi_*(\mu^u_x) = \nu^u_{\pi(x)}$ for $\mu-$almost every $x \in M.$ This property is equivalent to the so-called $u-$invariance of $\{\mu^c_x\}_{x \in M}$ (conditional measures of $\mu$ along central foliation) under unstable holonomies. See Proposition  \ref{gibbs-uinvariant} for the details.
 
 We prove the following main theorem:

\begin{main}  \label{u-invariantprinciple} (Entropy criterium for $u-$invariance)
Let $f$ be a $C^2-$partially hyperbolic diffeomorphism satisfying H1, H2 and H3. Suppose $\mu$ be an $f-$invariant probability and $\nu := \pi_*  \mu.$ Then, $h_{\mu} (f, \mathcal{F}^u) \leq h_{\nu} (f_ c)$ and equality occurs if and only if  $\mu \in Gibb^u_{\nu} (f).$
\end{main}

Observe that in the above theorem, we do not assume any hypothesis on the measures $\mu$ and $\nu$ to obtain the inequality.

As a corollary we may obtain the following invariance principle.

\begin{corollary} \label{exp-invariantprinciple}
Let $\mu$ and $f$ be as in Theorem \ref{u-invariantprinciple}. If all the central Lyapunov exponents of $\mu$  are non-positive almost everywhere, then $\mu$ is $u-$invariant.
\end{corollary}
 
Finally let us give another corollary of our main theorem which will be used in section \ref{proofhighentropy}.
\begin{corollary} \label{onedimensional}
Let $\mu$ and $f$ be as in Theorem \ref{u-invariantprinciple}. If the central foliation is one dimensional then $\mu \in Gibb^u_ {\nu}(f)$ if and only if $h_{\mu}(f) = h (\mu, \mathcal{F}^u).$
\end{corollary}
We also add another corollary which is a bit more technical, however useful in the development of the results using the invariance principle. Let $M, N$ be compact manifolds, 
$$
f: M \times N \rightarrow M \times N,  (x, \theta) \rightarrow (A(x), f_{x}(\theta)),
$$
where $A$ is an Anosov diffeomorphism. Fix any  $A-$invariant probability $\nu.$  

\begin{corollary} \label{invariancelimit}
Let $f_n$ be as above and $\mu_n,$ $f_n-$invariant such that $f_n \rightarrow f, \mu_n \rightarrow \mu.$ If $\mu_n \in Gibb^u_{\nu}(f_n)$ then $\mu \in Gibb^u_{\nu}(f).$  
\end{corollary}
We thank an anonymous referee to point out this corollary of our main result.
\subsection{Rigidity of high entropy measures in partial hyperbolic setting}

Let $M$ be a  smooth manifold. Denote by $\SPH_1(M)$ the set of
$C^2$ partially hyperbolic diffeomorphism $f$ on $M$ with {\bf one-dimensional}
central bundle satisfying hypothesis (H1), (H2) and (H3) plus accessibility property.

We remark that if $M$ is a closed orientable 3-manifold and $f$ is partially hyperbolic with compact center manifolds and $E^s, E^c$ and $E^u$ are orientable then $f_c$ is conjugate to a hyperbolic toral automorphism (See Theorem 3 in \cite{HHTU}) .

A special class of partially hyperbolic diffeomorphisms in $\SPH_1(M)$ are those which we denote by rotation type.
$f\in\SPH_1(M)$ is of \emph{rotation type} if there exists an isometric continuous action of $\mathbb{S}^1$ into $M$ which commutes with $f.$ That is, $\rho_{\theta} : M \rightarrow M, \theta \in \mathbb{S}^1$ with $f \circ \rho_{\theta} =\rho_{\theta} \circ f.$ Any rotation type partially hyperbolic diffeomorphism admits a unique measure of maximal entropy and its center Lyapunov exponent vanishes almost everywhere. This is the non-generic case (in $SPH_1(M)$) where the unique measure of maximal entropy is non-hyperbolic (see theorem \ref{dichotomy}).  It has been shown in \cite{HHTU} that for every $f\in \SPH_1(M)$ which is not rotation type, $f$ admits only
finitely (strictly larger than one) many ergodic maximal measures $\mu_1^+,\cdots, \mu_{k(+)}^+$ and $\mu_1^-,\cdots, \mu_{k(-)}^-$,
where $\mu_i^+$ has center exponent positive for $1\leq i \leq k(+)$, and $\mu_i^-$ has negative
center exponent for $1\leq i \leq k(-)$. 

\begin{theorem} \label{dichotomy} \cite{HHTU}
Suppose $f\in \SPH_1(M)$, then $f$ admits finitely many ergodic measures of maximal entropy. There are two possibilities:
\begin{enumerate}
\item $f$ is rotation type and has a unique entropy maximizing measure $\mu$. The central Lyapunov exponent $\lambda_c(\mu)$ vanishes and $(f, \mu)$ is isomorphic to a Bernoulli shift,

\item  $f$ has more than one ergodic entropy maximizing measure, all of which with non vanishing central Lyapunov exponent. The central Lyapunov exponent $\lambda_c(\mu)$ is nonzero and $(f, \mu)$ is a finite extension of a Bernoulli shift for any such measure $\mu.$ Some of these measures have positive central exponent and some have negative central exponent. \end{enumerate}

\end{theorem}

A more precise description of the ergodic maximal measures of the latter situation of the last theorem was obtained in \cite{UVY}.
The complement of rotation type diffeomorphisms is $C^1$ open and $C^{\infty}-$dense,  where we prove that high entropy
ergodic measures are hyperbolic.
\begin{main}\label{main.nonunihyp}

Suppose $f\in\SPH_1(M)$ is not rotation type, then there is $\vep >0$ and $\lambda_0>0$ such
that for every ergodic invariant probability measure $\mu$ of $f$ with entropy larger
than $h_{top}(f)-\vep$, its center exponent satisfies $|\lambda^c(\mu)|>\lambda_0$.
\end{main}

Theorem~\ref{main.nonunihyp} is a corollary of the following main result.

\begin{main}\label{main.converging}
Let $f\in \SPH_1 (M)$ which is not a rotation type, and $\{\mu_n\}_{n=1}^\infty$ be a sequence of
ergodic probability measure of $f$ such that $\lim_{n\to \infty} h_{\mu_n}(f)=h_{top}(f)$.
Suppose $\mu_n$ converges to $\mu$ in the weak-* topology and all $\mu_n$  have non-positive center
exponent, then $\mu$ is a combination of $\mu^-_1,\cdots, \mu^-_{k(-)}$.
\end{main}

\section{Quest for (non-)hyperbolic measures} \label{quest}
We would like to mention that Theorem \ref{main.nonunihyp} sheds light on some questions and conjectures in smooth ergodic theory.

\begin{conjecture} (D\'{i}az-Gorodetski \cite{DG})
In $\Diff^r(M), r \geq 1,$ there exists an open and dense subset  $\mathcal{U} \subset \Diff^r(M) $ such that  every $f \in \mathcal{U}$ is either uniformly hyperbolic or has an ergodic non-hyperbolic invariant measure.
\end{conjecture}

By our result, for $C^1$ open and $C^{\infty}-$dense partially hyperbolic dynamics with one dimensional compact central leaves (forming a circle bundle)  one cannot look for  non-hyperbolic ergodic measures among measures of large entropy.

Let us mention that by Bochi, Bonatti and D\'{i}az \cite{BBD} result, there exists an open and dense subset  $\mathcal{U} \subset \SPH_1(M) \cap RT(M)$ such that any $f \in \mathcal{U}$  has an ergodic measure with positive entropy and zero central Lyapunov exponent. Here $RT(M)$ is the set of $C^1-$robustly transitive diffeomorphisms. By our theorem, these non-hyperbolic ergodic measures can not have high entropy.

J. Buzzi  [section 1.5, \cite{Buz}] posed questions about abundance of hyperbolicity (of measures) for typical partially hyperbolic dynamical systems with one dimensional central bunde.  Our result  gives some partial answer to his questions too.

We would like also recall a recent result of D\'{i}az-Gelfert-Rams \cite{DGR}. They study transitive step skew -product maps modeled over a complete shift whose fiber maps are circle maps. They focus on non-hyperbolic measures (with zero fiber exponent) and prove that such measures are approximated in weak$-*$ and entropy by hyperbolic measures.  In the proof of their theorem 3, they consider three cases for the variational principle where the first case is:
$h_{top} = sup_{\mu \in \mathcal{M}_{erg, 0}} h_{\mu}$ where $ \mathcal{M}_{erg, 0}$ is the subset of ergodic measures of step skew product with vanishing fiber exponent. By our result this first case will not occur. We will not give the rigorous proof of this fact here and it will appear elsewhere. This gives a more accurate information for their study. Still finding the value $ sup_{\mu \in \mathcal{M}_{erg, 0}} h_{\mu}$ is interesting.

We would like to mention that in general setting of partially hyperbolic diffeomorphisms with one dimensional central bundle, it is not clear whether high entropy ergodic measures inherit the hyperbolicity of ergodic measures of maximal entropy (whenever all of ergodic measures of maximal entropy are hyperbolic).
For $C^{1+ \alpha}-$diffeomorphisms in the homotopy class of Anosov diffeomorphisms of $\mathbb{T}^3$, a similar argument to Theorem 5.1 in \cite{ures} shows that high entropy measures are hyperbolic.

\subsection{Less regularity}
Although throughout this article we always assume the diffeomorphisms to be $C^2$, Theorems~\ref{main.nonunihyp} and
~\ref{main.converging} also hold for $C^{1+\alpha}$ diffeomorphisms. In fact, the only place 
where $C^2$ hypothesis used is where we use the Lipschitzness of unstable holonomy inside center-unstable plaques (see \cite{LY}) to conclude that when center Lyapunov exponents are non-positive, the entropy of a measure is equal to the entropy along the unstable foliation (See the Proof of Theorem \ref{main.converging}.)

More precisely, the $C^2$ hypothesis is used to obtain Lipschitz holonomy of Pesin unstable lamination inside the center
unstable set (see \cite{LY}[Section 4.2]). For $f\in \Diff^{1+\alpha} (M)$, it has been shown in \cite{Brown} that the strong unstable foliation restricted to each center unstable leaf is Lipschitz,  then one may repeat the proof of \cite{LY}.

\section*{acknowledgment}
A. T was in a research period at Universit\'{e} Paris-Sud (thanks to support of FAPESP-Brasil:2014/23485-2, CNPq-Brasil) and would like to thank hospitality of Laboratoire de Topologie and in particular, Sylvain Crovisier and J\'{e}r\^{o}me Buzzi for many useful conversations. 
J.Y. was partially supported by CNPq, FAPERJ, and PRONEX.

\section{ Preliminaries } \label{preliminaries}

A partially hyperbolic diffeomorphism is called {\it accessible} if  one can join any two points in the manifold by a
path of finitely many  components which is piecewise tangent to either $E^s$ or $E^u.$

In general it is not true that there is a foliation tangent to
$E^c$. Indeed, there may be no foliation tangent to $E^c$ even if  $\dim E^c =1$ (see \cite{HHU}). We shall say that $f$ is  \emph{dynamically coherent} if there exist invariant foliations $\cF^{c\s}$ tangent to $E^{c\s}=E^c \oplus E^\s$  for $\s=s,u$. Note that by taking the intersection of these foliations we obtain an invariant foliation $\cF^c$ tangent to $E^c$ that subfoliates $\cF^{c\s}$ for $\s=s,u$. Observe that $\cF^{\s}$ also subfoliates $\cF^{c\s}$ for $\s \in \{s, u\}.$

For any dynamically coherent $f$ and any two points $x, y$ with $y \in \mathcal{F}^u(x),$  there is a neighbourhood $U_x$ of $x$ in $\mathcal{F}^c(x)$ and a homeomorphism $H^u_{x, y} : U_x \rightarrow \mathcal{F}^c(y)$ such that $H^u_{x, y}(x) =y$ and $H^u_{x, y}(z) \in \mathcal{F}^u(z) \cap \mathcal{F}^c_{loc} (y).$ Similarly, one may define local stable holonomies $H^s_{x, y}$ for $y \in \mathcal{F}^s(x).$

We say $f$ admits global unstable holonomy if for any $y \in \mathcal{F}^u(x)$ the holonomy is defined globally $H^u_{x, y} : \mathcal{F}^c(x) \rightarrow \mathcal{F}^c(y).$ Similarly we define the notion of global stable holonomy and $f$ admits global holonomies when it admits global stable and unstable holonomies.

There are many robust examples of partially hyperbolic diffeomorphisms which are dynamically coherent.
The simplest construction goes as follows. Start with a hyperbolic toral automorphism $A : \mathbb{T}^2\to \mathbb{T}^2$,
and then let $f_0\in \Diff(\mathbb{T}^2\times S^1)$ be a skew product map such that
$$f_0((x,\theta))= (A(x),\theta)
$$
Then $f_0$ is a partially hyperbolic diffeomorphism and so is every $f$ in a $C^1$ neighborhood of $f_0$.
Moreover, it follows from general results in \cite{HPS} that $f$ is indeed dynamically coherent and admits global holonomies.
A generalization of the above examples are fibered partially hyperbolic dynamics (See Avila-Viana-Wilkinson and Hirsch-Pugh-Shub \cite{HPS}). They are examples of dynamically coherent dynamics admitting global holonomies. A partially hyperbolic diffeomorphism $f: M \rightarrow M$ is fibered partially hyperbolic dynamics if there exists a continuous fiber bundle $\pi : M \rightarrow B$ with fibers modeled by compact manifold such that $\pi^{-1} (b)$ is a center leaf of $f$ for any $b \in B.$

Let  $M_c:= M / \mathcal{F}^c$ be the quotien space and denote by $f_c$ the induced dynamics, i.e; $f_c \circ \pi = \pi \circ f$ where $\pi: M \rightarrow M_c$ is the natural projection. $f_c$ is a {\it topological Anosov homeomorphism} (See \cite{Vi} and section 2.2 of \cite{YV}).  We assume that $f_c$ is transitive (hypothesis H3 in section 2), which is the case for all known examples.   A transitive topological Anosov homeomorphism shares many similar properties with Anosov diffeomorphisms. For instance, it has a pair of topological foliations which play the same role of stable/unstable foliations, there exist Markov partitions (for Markov partitions in the topological Anosov setting see \cite{Hi}.) and its unique measure of maximal entropy can be obtained by the Margulis method. We give a more precise description in what follows:

An Anosov homeomorphism of $M_c,$ by definition  admits two invariant topological foliations $\cW^s$ and $\cW^u$ with similar dynamical properties as in the diffeomorphism case. The leaves are topological submanifolds and $$\cW^s(\xi) = \bigcup_{n} f_c^{-n} \cW^s_{\epsilon}(f_c^n(\xi)), \cW^u(\xi) = \bigcup_{n} f_c^{n} \cW^u_{\epsilon}(f_c^{-n}(\xi))$$
where $$\cW^s_{\epsilon}(\xi) = \{ \eta \in M_c: d_c(f_c^n (\xi),f_c^n (\eta) ) \leq \epsilon\} $$
$$ \cW^u_{\epsilon}(\xi) = \{\eta \in M_c: d_c(f_c^{-n} (\xi),f_c^{-n} (\eta) ) \leq \epsilon\}   $$
and $d_c$ is a distance in $M_c:$
$$
d_c(\xi,\eta) := \sup_{x \in \xi} \inf_{y \in \eta} d(x, y) + \sup_{y \in \eta} \inf_{x \in \xi} d(x, y)
$$ for any $\xi, \eta \in M_c.$

As we will assume that $f_c$ is transitive throughout this article, the quotient space $M_c$ is homotopic to a $d-1$ dimension torus $\TT^{d-1}$.
We use $\cF^{\s}$, $\s=s,c,u$ to denote the invariant foliations of $f$, and $\cW^i$, $i=s,u$ to
denote the stable and unstable foliation of $f_c$. As a summary, we have: 
 $f_c$ is a transitive topological Anosov homeomorphism. It
admits a Markov partition and a unique measure of maximal entropy $\nu$. 



\subsection{Nonlinear cocycle and Invariance Principle} \label{invariance1}

Any $f: M \rightarrow M$ satisfying (H1), (H2) and (H3) can be considered as a smooth (nonlinear) cocycle over $f_c: M_c \rightarrow M_c$ with global holonomies.  We denote by $M_{x_c}$ the fiber $\pi^{-1}(x_c)$ which is a center leaf and $f_{x_c} : M_{x_c} \rightarrow M_{f_c(x_c)}$ is the restriction of $f$ on $M_{x_c}.$ By hypothesis we have two families of stable and unstable holonomies: for any $y_c \in W^u(x_c),$ then $H^u_{x_c , y_c} : M_{x_c} \rightarrow M_{y_c}$ satisfies:
\begin{itemize}
\item $H^u_{y_c, z_c} \circ H^u_{x_c, y_c} = H^u_{x_c, z_c}$ and $H^u_{x_c, x_c} = id$
\item $f_{y_c} \circ H^u_{x_c, y_c} = H^u_{f_c(x_c), f_c(y_c)} \circ f_{x_c}$
\end{itemize}
A crucial feature of the above holonomy map is that it is Lipschitz when $f$ is $C^2.$
The consequences of this regularity are explored in the work of Ledrappier-Young \cite{LY} and it is used in our paper.
If the central foliation is one-dimensional then  the Lipschitz property holds for the lower regular diffeomorphisms, for instance, $f \in C^{1+ \alpha}.$ See \cite{PSW,Brown} for other technical conditions to guarantee the Lipschitz property of stable/unstable holonomies inside center-stable/center-unstable manifolds.

Any invariant measure (by $f$) $\mu$ projects down to $\nu:= \pi_* \mu$ which is invariant by $f_c.$
By Rokhlin's \cite{R} disintegration theorem, there exist a system of conditional measures $\{\mu^c_{x_c}\}$ such that $$\mu = \int \mu^c_{x_c} d \nu(x_c).$$

We seek for a criterium which implies the invariance of conditional measures by $u-$holonomy (or $s-$holonomy), i.e.
\begin{equation} \label{determinism}
\mu^c_{y_c} = (H^u_{x_c, y_c})_{*} \mu^c_{x_c}
\end{equation}
 for any $y_c \in W^u(f_c, x_c)$ for $x_c, y_c$ belonging to a full $\nu$ subset of $M_c.$ Similarly we define $s-$invariance of the conditional measures. This invariance of conditional measures by holonomies is called invariance principle by Avila and Viana \cite{AV}.

The simple but fundamental observation is that instead of verifying  $u-$invariance  of $\{\mu^c_{x_c}\}$ we may verify equivalently $c-$invariance of conditional measures on local unstable plaques $\{ \mu^u_{x} \}$. This can be rewritten as $$\pi_* \mu^u_{x} = \nu^u_{\pi(x)}.$$ We denote by $Gibb^u_{\nu} (f)$ any measure satisfying the above condition. See section \ref{proofIP} for more precise definitions.

The choice of local unstable plaques is based on the fact that the quotient dynamics $f_c$ admits a natural partition (Markov partition) and by means of such partition we are able to define a measurable partition sub-ordinated to unstable foliation. The atoms of such measurable partition are called unstable plaques.

The above observation about the equivalence between $u-$invariance and $c-$invariance and other abstract measure theoretical results are proved in Section \ref{measurabletoolbox}.

\subsection{Invariance principle with vanishing exponents} \label{invariance2}
Using the above corollary for $f$ and $f^{-1}$ one can conclude that if the central Lyapunov exponents are zero then $\mu$ is both $u$ and $s-$invariant. That means, there are two systems of conditional measures one of them  $s-$invariant and another $u-$invariant. Observe that up to now, the conditional measures are measurable objects: They are defined on a full measurable subset and varies measurably. If we suppose that $\nu:= \pi_* \mu$ has {\bf local product structure} (as defined in measurable toolbox section \ref{measurabletoolbox}) then one can prove that there exists a system of conditional measures defined {\bf everywhere} and  the conditional measures depend continuously on the base point. This last passage makes a subtle use of Hopf type argument, See Avila-Viana's invariance principle (See theorem (D) and Proposition 4.8 in \cite{AV}). So we would like to remark that our theorem and the Hopf type argument given by Avila-Viana reveals a different and shorter proof of the Invariance principle. Moreover, our approach makes the role of entropy much more clear.

We emphasize that we prove and apply our result  in the case of partially hyperbolic diffeomorphisms (as we see in section \ref{proofhighentropy}).  Avila-Viana proved  $u-$invariance principle in the more general setting and in the same spirit of Ledrappier result for abstract sub-sigma algebras of the base space. See Theorems (B) in \cite{AV}.


\section{A measurable toolbox} \label{measurabletoolbox}

In this section we develop an abstract measurable toolbox which deals with disintegration of measures and their properties mostly related to invariance with respect to the holonomy of foliations. {\bf The notations are similar to the dynamical ones, but no dynamics is assumed here.}
Let $(A:=[0, 1]^{c+u}, \mu^{cu})$ be the unit square equipped with a probability measure and $\mathcal{F}^c , \mathcal{F}^u$ be a pair of  transversal foliations of $A$ with compact leaves and dimension of leaves respectively $c$ and $u$.

We assume the following topological  product structure: There exists a continuous injective and surjective map $Q(. , .) : \mathcal{F}^u(x_0) \times \mathcal{F}^c(x_0) \rightarrow [0, 1]^{c+u}$ such that $Q(x, y) = \mathcal{F}^c(x) \cap \mathcal{F}^u(y).$

\begin{definition} \label{def:mensuravel}
We say that a partition $\mathcal P$ is measurable (or countably generated) with respect to $\mu$ if there exist a measurable family $\{A_i\}_{i \in \mathbb N}$ and a measurable set $F$ of full measure such that
if $B \in \mathcal P$, then there exists a sequence $\{B_i\}$, where $B_i \in \{A_i, A_i^c \}$ such that $B \cap F = \bigcap_i B_i \cap F$.
\end{definition}

 Let $\mathcal P$ be a measurable partition of a compact metric space $M$ and $\mu$ a borelian probability. Then, by Rokhlin's theorem \cite{R},  there exists a disintegration by conditional probabilities for $\mu$.
 The foliations $\mathcal{F}^{u,c}$ will be considered as measurable partitions and we denote by $\{\mu_{x}^c\}$ and $\{\mu^u_{x}\}$ the system of conditional probability measures along  $\mathcal{F}^c$ and $\mathcal{F}^u:$
$$\mu^{cu} = \int_{A} \mu^u_{x} d\mu^{cu} (x) = \int_{A} \mu^{c}_x d \mu^{cu}(x)$$
where $\mu_ {x}^{u}$ (resp. $\mu^c_x$) is a probability measure depending only on the leaf $\mathcal{F}^{u} (x)$ (respt. $\mathcal{F}^c(x)$).

Another equivalent way to write the disintegration equation (along $\mathcal{F}^c$) above  is to consider the quotient space $A/ \mathcal{F}^{c}$ equipped with the quotient measure $\tilde{\mu}^{cu}:= \pi_*(\mu^{cu})$ where $\pi: A \rightarrow A/\mathcal{F}^{c}$ is the canonical projection. We can write
$$
 \mu^{cu} =  \int_{A/\mathcal{F}^c} \mu_{P}^c d \tilde{\mu}^{cu} (P)
$$
where $\mu^c_{P}$ is the conditional probability measure on a typical leaf of $\mathcal{F}^c$. By definition,  for any integrable function $\phi: M\to \mathbb{R}$, we have
$$\int_{A} \phi d \mu^{cu} = \int_{A/\mathcal{F}^c} \int_{P} \phi(x) d \mu_{P}^c (x) d \tilde{\mu}^{cu}(P).  $$

The product structure of the pair of foliation above, permits us to define holonomy maps $H^u$ and $H^c$ respectively between leaves of $\mathcal{F}^c$ and $\mathcal{F}^u.$

We say a system of disintegration $\mu^c$ is \emph{$u-$invariant} if $\mu_y^{c} = (H^u_{x, y})_{*} \mu_x^c$ for
 $x,y$ belong to a $\mu$ full measure subset, where $H_{x, y}^u$ is the $u-$holonomy map between $\mathcal{F}^c(x)$ and $\mathcal{F}^c(y)$ induced by the foliation $\mathcal{F}^u$. Similarly we define \emph{$c-$invariance} of $\{\mu_{y}^u\}$ by $\mu_{y}^u = (H^c_{x, y})_{*} \mu^u_{x}.$
\begin{lemma} \label{product}
If $\{\mu_{x}^c\}$ is $u-$invariant then  $\{\mu_x^u\}$ is $c-$invariant and $\mu^{cu} = Q_* (\mu^u_{x_0} \times \mu^c_{x_0})$ for any typical point $x_0.$
\end{lemma}
By the definition of conditional measures and $u-$invariance of $\mu^c$ we have
$$
 \int \phi d \mu = \int_{\mathcal{F}^u(x_0)} \int_{\mathcal{F}^c(x_0)} \phi \circ H^u_{x_0, z}(x) d \mu^c_{x_0}(x) d \tilde{\mu} (z)
$$
for any continuous function $\phi$ where $\tilde{\mu}$ is the quotient measure on the quotient space identified with $\mathcal{F}^u(x_0).$
Using Fubini's theorem and the fact that $H^u_{x_0, z}(x) = H^c_{x_0, x} (z)$ for any $x \in \mathcal{F}^c(x_0) , z \in \mathcal{F}^u(x_0)$ we obtain
$$
\int \phi d \mu = \int_{\mathcal{F}^c(x_0)} \int_{\mathcal{F}^u(x_0)} \phi \circ H^c_{x_0, x}(z) d \tilde{\mu}(z) d \mu^c_{x_0}(x).
$$
By essential uniqueness of disintegration, the above equality shows that the system of conditional measures $\{ \mu^u\}$ satisfies $\mu^u_{x} = H^c_{x_0, x} \tilde{\mu}$, which implies the $c-$invariance.
The last claim of the lemma is direct from definition and Fubini theorem.

\subsection{Disintegration along three (coherent) foliations} \label{x0}
Now we consider a probability measure on the unit cube $K = [0, 1]^{s+c+u}$ equipped with probability $\mu$ and consider   three transverse foliations $\mathcal{F}^{s}, \mathcal{F}^{c}, \mathcal{F}^{u}$ (we call them stable, central and unstable foliation) and assume the following coherence property:  there exist two more foliations $\mathcal{F}^{cu}$ (center-unstable foliation) and $\mathcal{F}^{cs}$ (center-stable foliation) such that subfoliate respectively  into $\mathcal{F}^{c}, \mathcal{F}^{u}$  and $\mathcal{F}^{c}, \mathcal{F}^{s}.$ Moreover, inside any leaf of $\mathcal{F}^{cs}$ the leaves of $\mathcal{F}^{c}$ and $\mathcal{F}^{s}$ have product structure. Similarly we assume product structure inside leaves of $\mathcal{F}^{cu}.$

We have  two holonomy maps called unstable holonomy and stable holonomy. The unstable holonomy is defined between any two central leaves inside a center-unstable leaf and similarly we define stable holonomy. Observe that in this section we are not assuming any dynamical property for these foliations and just use the names that will be used later.


We make some useful geometrical identifications of the quotient spaces. Let $x_0 \in K$ which we fix from now on.  By the coherence hypothesis the quotient spaces $K/\mathcal{F}^s$ and $K/\mathcal{F}^u$ may be identified by $\mathcal{F}^{cu}_{x_0}$ and $\mathcal{F}^{cs}_{x_0}.$

The quotient by central foliation $K /\mathcal{F}^c$ is a compact metric space and it admits  two topological foliations $W^u$ and $W^s.$  Let $\pi: K \to \tilde{K}=K/ \mathcal{F}^c$  be the canonical projection and  $\nu:= \pi_* (\mu)$.

We have two following important properties:
\begin{itemize}
\item $\pi(\mathcal{F}^{cu} (x)) = W^u(\pi(x))$
\item $\pi(\mathcal{F}^{cs} (x)) = W^s(\pi(x))$.
\end{itemize}
So we may identify $K/ \mathcal{F}^{cu}$ and $K/ \mathcal{F}^{cs}$ respectively with $W^s(\pi (x_0))$ and $W^u(\pi(x_0)).$ We also may identify $\tilde{K} / W^u$ and $\tilde{K} / W^s$ respectively with $W^s(\pi(x_0))$ and $W^u(\pi(x_0)).$
With a bit of abuse of notations, for $t \in W^s(\pi (x_0))$ we denote by $\mathcal{F}_{t}^{cu}$ as the center-unstable plaque corresponding to $t.$

In what follows we study conditional measures along various foliations, by
$$\{\mu_{x}^{*}\}, * \in \{s, c, u, cs, cu\}$$
we mean the conditional probabilities along leaves of foliation $\mathcal{F}^{*}.$ We also may disintegrate $\nu$ along foliations $W^u$ and $W^s$ obtaining two additional system of conditional measures $\nu_{\pi(x)}^u$ and $\nu_{\pi(x)}^s.$

\begin{lemma} \label{quotient}
$\pi_* (\mu_{x}^{cu}) = \nu^u_{\pi(x)}$ for $\mu$ almost every $x.$
\end{lemma}

\begin{proof}

By definition $\mu = \displaystyle{\int} \mu^{cu} d \tilde{\mu}$ where $\tilde{\mu}$ is the probability on the quotient $K / \mathcal{F}^{cu}$ which is identified with $W^s(\pi(x_0)).$ We also have
$\nu = \displaystyle{\int} \nu^u d \tilde{\nu}$ where the quotien measure $\tilde{\nu}$ is defined on  $\tilde{K} / W^u = W^s(\pi(x_0)).$  It is not difficult to see that  $\tilde{\nu} = \tilde{\mu}.$ Indeed taking any measurable subset $S \subset  W^s(\pi(x_0))$ we have:
$$
 \tilde{\mu} (S) = \mu( \bigcup_{t \in S} \mathcal{F}^{cu}(t))
$$
and
$$
\tilde{\nu} (S) = \nu(\bigcup_{t \in S} W^u(t)) = \mu( \pi^{-1} (\bigcup_{t \in S} W^u(t)) ) = \mu( \bigcup_{t \in S} \mathcal{F}^{cu}(t)).
$$

So, by definition
\begin{align*}
\int_{\tilde{K}/ W^u } \nu^u d \tilde{\nu} = \nu =& \pi_*{\mu}
= \pi_* (\displaystyle{\int}_{K/ \mathcal{F}^{cu}}\mu^{cu} d \tilde{\mu} ) \\=& \int_{K/ \mathcal{F}^{cu}} \pi_{*} \mu  d \tilde{\mu} = \int_{\tilde{K}/ W^u } \pi_{*} \mu^{cu}  d \tilde{\nu}
\end{align*}

By essential uniqueness of disintegration we conclude the proof of lemma.
\end{proof}

\begin{proposition}  \label{gibbs-uinvariant}

    $\nu_{\pi(x)}^u = \pi_* (\mu_{x}^u)$ holds for $\mu$ almost every $x$ if and only if $\mu^c$ is $u-$invariant.
\end{proposition}

By coherent hypothesis we may speak about $H^u$ holonomy between two central leaves and consequently $\{\mu^c\}$ being $u-$invariant makes sense.
\begin{proof}
Let us prove ``if" part: Observe that by essential uniqueness of disintegration the family of conditional measures $\mu_{x}^u$ almost everywhere coincides with the disintegration of $\mu^{cu}$ along $\mathcal{F}^u.$ The same statement for $\mu^c$ holds. That is, the disintegration of $\mu^{cu}$ along the central palques coincides with $\mu^c_x.$ Observe that for any center-unstable plaque $\mathcal{F}^{cu}(x)$, the quotient by central plaques can be identified by the unstable plaque $\mathcal{F}^u(x).$

By lemma \ref{product} and invariance of $\mu^c$ by $u-$holonomy we conclude that $\mu^u$ is invariant by $\mathcal{F}^c$ holonomy.  This yields that for any $D \subset \mathcal{F}^u(x)$ we have $$\mu^{cu}_{x} (\mathcal{F}^c(D)) = \int_{\mathcal{F}^u(x)} \mu^u_z (H_{x,z}^c (D)) d \tilde{\mu}^{cu} (z)  = \mu_{x}^u(D)$$
where $   \mathcal{F}^c(D) = \bigcup_{z \in D} \mathcal{F}^c(z)$, $H_{x, z}^c$ is the central holonomy map between two unstable leaves and $\tilde{\mu}^{cu}$ the probability one the quotient space $\mathcal{F}^{cu}/\mathcal{F}^u$.
Observe that
$$
\mu^{cu}(\mathcal{F}^c(D)) = (\pi_* \mu^{cu}) (\pi(D)) = \nu^u_{\pi(x)} (\pi(D)).
$$
where the last equality comes from Lemma \ref{quotient}.
Comparing the above two equations we conclude that $\nu^u_{\pi(x)} = \pi_*(\mu^u_x)$.

To prove the ``only if" part, just observe that  $\nu^u_{\pi(x)} = \pi_* \mu^u_x$ implies 
$\mu^u$ is $c-$invariant and by lemma \ref{product} we conclude that $\mu^c$ is $u-$invariant.

\end{proof}






\section{Proof of Entropy criterium for invariance principle} \label{proofIP}

 Throughout this section we prove Theorem \ref{u-invariantprinciple}. Recall that $\mu$ denotes an invariant probability measure by $f$ and that $\pi_*(\mu)=\nu$.
 Fix $\{A^i_c\}$ a Markov partition of $f_c$. Denote by $A^i=\pi^{-1}(A^i_c)$,
then $\cA=\{A^{i}\}$ is a partition on the manifold $M$. We may assume that the boundary of each element of this partition has zero $\mu-$measure.

Fix $i=1,\dots,k$, for any $x\in A^i$, denote by $W^u_{loc}(\pi(x))$
the unstable plaque contained inside $A^i_{c}$, and by $\cF^u_{loc}(x)$ (unstable plaque) the connected component  of
the unstable leaf of $\cF(x)$ which intersects $A^i$ and contains $x$. In the proof we use four measurable partitions:
\begin{itemize}
\item Central foliation $\mathcal{F}^c$ is a foliation by compact leaves and so it is a measurable partition. The conditional measures of $\mu$ along this partition are denoted by $\{\mu^c_x\};$
\item $\xi_c^u=\{W^u_{loc}(\pi(x)); \pi(x) \in M_c\}$ is a measurable partition of $M_c$ by unstable plaques of $f_c$. We may disintegrate $\nu = \pi_* \mu$ along this partition and the conditional measures are denoted by $\{\nu^{u}_{\pi(x)}\}$;
\item $ \pi^{-1}(\xi_c^u)$ is a measurable partition of $M$ by $\mathcal{F}_{loc}^{cu}$ (center-unstable) plaques. The corresponding conditional measures of $\mu$ are denoted by $\{\mu^{cu}_x\}$
\item and $\xi^u=\{\cF^u_{loc}(x); x\in M\}$ is a measurable partition of $M$ by unstable plaques of $f$ and $\{\mu^u_x\}$ stands for the system of conditional measures of $\mu$.

\end{itemize}

Considering the conditional measures of $\mu$ along different measurable partition introduced above we define a new category of measures which we call ``$u-$Gibbs relative to measure $\nu$" or just by $Gibb^u_{\nu}$ states.
\begin{definition}\label{d.u state}
We say $\mu$ is a \emph{$Gibbs^u_{\nu}$-state} if $\pi_* \mu = \nu$ and
for $\mu-$almost every $x\in M$,
$$\pi_*\mu^u_x=\nu^u_{\pi(x)}.$$
\end{definition}
We denote by $\Gibb_{\nu}^u(f)$ the set of Gibbs$^u_{\nu}$-states of $f$.
Observe that by Proposition \ref{gibbs-uinvariant} all measures in $\Gibb_{\nu}^u(f)$ have $u-$invariant disintegration along the central foliation.

Recall that a measurable partition $\eta$ for a map $f$ is \emph{increasing} if $f\eta \preceq \eta$. Then
all the three partitions $\xi_c^u$, $\pi^{-1}(\xi_c^u)$ and $\xi^u$ are increasing. It is easy to see that $\xi^u$ is finer than $\pi^{-1}(\xi_c^u).$



\subsection{Partial entropy along expanding foliations} \label{entropyfoliation}

In this section, we recall the general definition of partial entropy along the expanding foliation (See \cite{L1}, \cite{LY2} and \cite{Y}.)

Let $f$ be a diffeomorphism, we say a foliation $\cF$ is \emph{$f$-expanding} if:
\begin{itemize}
\item $\cF$ is invariant;
\item $f$ is expanding along $\cF$.
\end{itemize}

\begin{remark}
By unstable manifold theorem, the unstable foliation $\cF^u$ is an expanding foliation of $f$.
It deserves to observe that, although $f_c$ is only a topological Anosov homeomorphism, we may
still consider $\cW^u$ as an \emph{expanding foliation} of $f_c$. Indeed, we can use a conjugacy
to identify it with a linear Anosov diffeomorphism $A_0$ and the unstable foliation is preserved by the conjugacy.
\end{remark}

For any invariant probability measure $\mu$ of $f$, we say a measurable partition $\xi$ is \emph{$\mu$ adapted} (sub-ordinated)
to $\mathcal{F}$ if the following conditions are satisfied:
\begin{itemize}
\item there is $r_0>0$ such that $\xi(x)\subset B^{\cF}_{r_0}(x)$ for $\mu$ almost every $x$, where $B^{\cF}_{r_0}(x)\subset \cF(x)$ is a ball of $\cF(x)$ with radius $r_0$
\item $\xi(x)$ contains an open neighborhood of $x$ inside $\cF(x)$;
\item $\xi$ is increasing, that is, for $\mu$ almost every $x$, $\xi(x)\subset f(\xi(f^{-1}(x)))$.
\end{itemize}

Then the \emph{$\mu$ partial entropy along expanding foliation $\cF$} is defined by
$$h_{\mu}(f,\cF)=H_{\mu}(f^{-1} \xi \mid \xi).$$

\begin{remark}
It is easy to check that $\xi_c^u$ is $\nu$ adapted to foliation $\cW^u$ and $\xi^u$ is
$\mu$ adapted to $\cF^u$. Then, by the definition,
\begin{equation}\label{eq.partialentropy}
h_\nu(f_c,\cW^u)=H_\nu(f_c^{-1} \xi_c^u\mid \xi_c^u) \text{ and } h_\mu(f,\cF^u)=
H_\mu(f^{-1}\xi^u\mid \xi^u).
\end{equation}
\end{remark}



\subsection{Proof of Theorem A}
 We use the notations introduced in beginning of this section: $\{A^i\}$ is a partition of $M$ into finitely many domains. Each $A^i$ is partitioned into stable, unstable and central plaques with the coherence property. We use abstract results obtained in measurable toolbox (Section \ref{measurabletoolbox}) for each $A^i.$ For each $A^i$ fix $x_i \in A^i$ which plays the role of $x_0$ in subsection \ref{x0}.


To simplify the notations we use $\mathcal{F}^{cu}_{loc}(t)$ to denote the atom of partition $\pi^{-1}(\xi_c^u)$ containing $t \in A^i$.
By definition,
\begin{align} \label{entropy}
 h_{\mu} (f, \mathcal{F}^u) &= \int_M -\log \mu_{z}^u ( f^{-1} \xi^u (z)) d \mu(z)  \\
 & = \sum_ i \int_{W^s_{loc}(\pi(x_i))} \int_{\mathcal{F}^{cu}_{loc}(t)}  - \log \mu_{z}^u (f^{-1} \xi^u (z)) d \mu_{t}^{cu} (z)  d \tilde{\mu} (t) \\
 &=  \sum_ i \int_{W^s_{loc}(\pi(x_i))} \int_{\mathcal{F}^{cu}_{loc}(t)}  - \log \mu_{z}^u (f^{-1} \xi^u (z)) d \mu_{t}^{cu} (z)     d \tilde{\nu} (t)
\end{align}

where the sum above is over all   $A^i$ and $f^{-1} \xi^u (x)$ stands for the element of the partition $f^{-1} \xi^u$ which contains $x.$ The second equality comes from the disintegration
$$
\mu= \int \mu^{cu} d \tilde{\mu}.
$$
For the third equality we identify the quotient of $A^i$ by the center-unstable plaques, with the stable plaque of $\pi(x_i)$
and recall that  the quotient measure $\tilde{\mu}$ can be identified with the quotient measure $\tilde{\nu}$ where $\displaystyle{\nu = \int \nu^u d \tilde{\nu}}$ (See the proof of Lemma \ref{quotient}.)

Now observe that $f_c^{-1} (\xi^u_{c})$ induces a partition on each element of $\xi^u_{c}.$ Taking pre images by $\pi$ we conclude that each $\mathcal{F}_{loc}^{cu}(t)$ is partitioned into finitely many subsets $\mathcal{F}_{loc}^{cu}(t) = \bigcup_{j} B_j$ where for each $j,$ $\pi(B_j)$ is an atom of $f_c^{-1} (\xi^u_{c}).$

\begin{figure}
\includegraphics[scale=0.3]{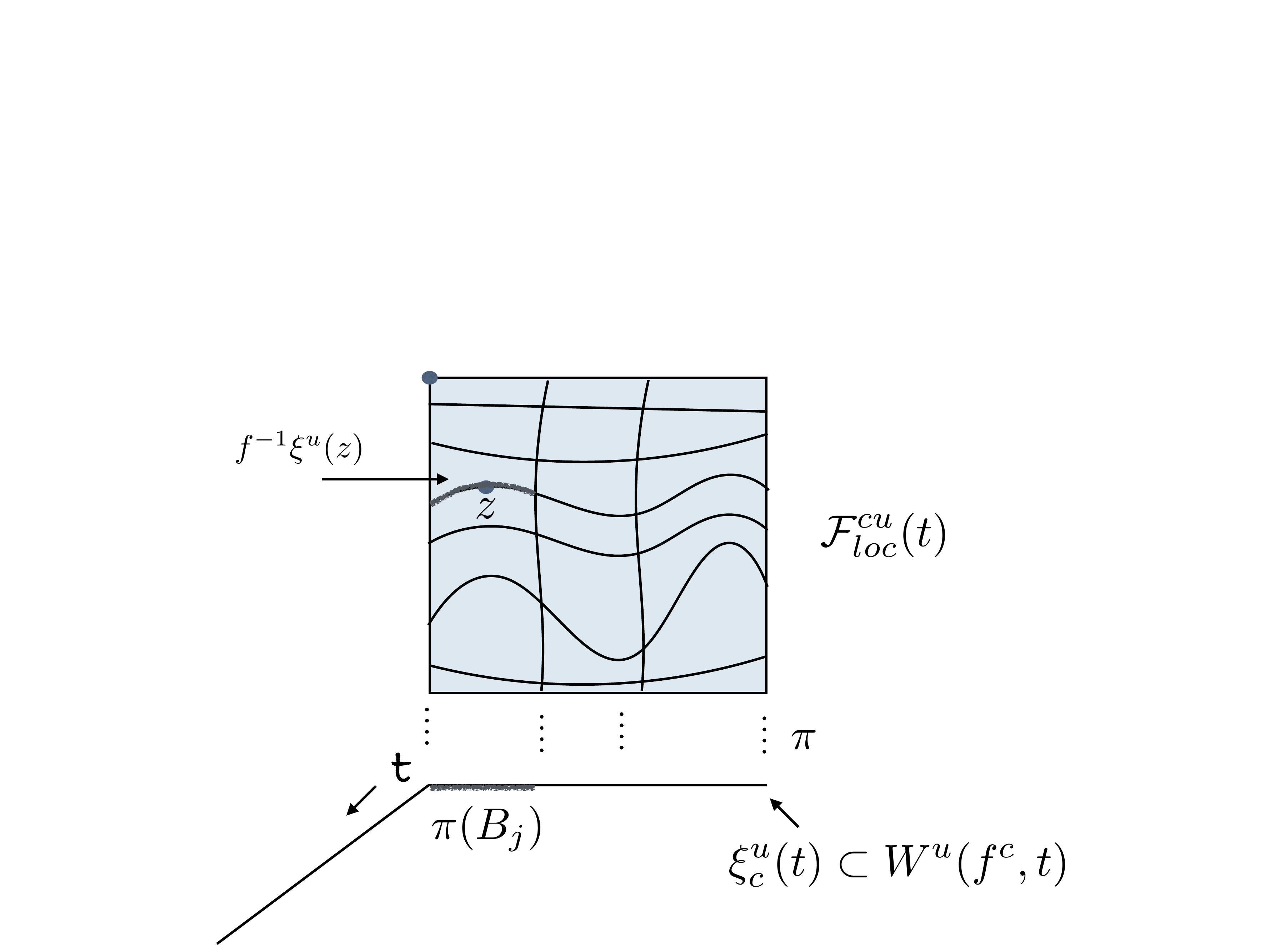}
\centering
\caption{}
\end{figure}

\begin{Claim} For any $t \in W^s(\pi(x_i))$ we have:
 $$
  \int_{\mathcal{F}_{loc}^{cu}(t)}  - \log \mu_{z}^u (f^{-1} \xi^u (z)) d \mu_{t}^{cu} (z) \leq \sum_j - \nu_t^u (\pi(B_j)) \log (\nu_t^u (\pi(B_j)).
  $$
\end{Claim}


To prove the above claim (See the above figure), first by lemma \ref{quotient} we obtain that for each $z \in B_j$
$$\nu^u_{t} (\pi(B_j)) = \mu_t^{cu} (B_j) = \int_{\mathcal{F}^{c}(t)} \mu^u_{\theta} (B_j)  d \tilde{\mu}_{t}^{cu} (\theta),
$$
where $\tilde{\mu}_t^{cu}$ is the measure on $\mathcal{F}^{cu}_{loc}(t)/\xi^u\approx \mathcal{F}^c(t)$.

Take $g(x)= -x \log x$ for $x > 0$ and apply Jensen inequality to obtain:
\begin{align} \label{cimabaixo}
\int_{\mathcal{F}^{c}(t)} g (\mu_ {\theta}^u (B_j))  d \tilde{\mu}_{t}^{cu} (\theta) \leq g (\int_{\mathcal{F}^{c}(t)} \mu^u_{\theta} (B_j)  d \tilde{\mu}_{t}^{cu} (\theta) ) = - \mu_t^{cu} (B_j) \log (\mu_t^{cu}(B_ j)).
\end{align}


Now,

\begin{align*}
\int_{\mathcal{F}^{cu}(t)}  - \log \mu_{z}^u (f^{-1} \xi^u (z)) d \mu_{t}^{cu} (z)  = \\
=\sum_j \int_{\mathcal{F}^{cu}(t)}  - \mathcal{X}_{B_j}(z)\log \mu_{z}^u (f^{-1} \xi^u (z)) d \mu_{t}^{cu} (z)\\
\sum_j \int_{\mathcal{F}^{c}(t)} \int_{\mathcal{F}^u_{loc}(\theta)} -\mathcal{X}_{B_j}(z) \log \mu_{z}^u (f^{-1} \xi^u (z)) ) d\mu^u_{\theta}(z) d \tilde{\mu}_{t}^{cu} (\theta) =
\\
\sum_j \int_{\mathcal{F}^{c}(t)} -\mu^u_\theta(B_j) \times (\log \mu_{\theta}^u (B_j) ) d \tilde{\mu}_{t}^{cu} (\theta) =
\\
= \sum_j  \int_{\mathcal{F}^{c}(t)} g( \mu^u_{\theta} (B_j) ) d \tilde{\mu}_{t}^{cu} (\theta)  \leq^{\text{by}(\ref{cimabaixo})}
\\
 \sum_j  -\mu_t^{cu} (B_j) \log (\mu_t^{cu} (B_j)) =  \\
 \sum_j -\nu_t^u (\pi(B_j)) \log (\nu_t^u (\pi(B_j))
\end{align*}

and we completed the proof of the claim.

Now taking integral from the both sides of the expressions in the claim with respect to $\tilde{\nu}$  and using \eqref{entropy} and summing over all $A^i$ we obtain that
$$h_{\mu} (f,\mathcal{F}^u) \leq H_\nu(f_c^{-1}\xi^u_c\mid \xi^u_c) = h_{\nu} (f_c).$$

Indeed, $h_{\nu} (f_c) = H_ {\nu} (f_c^{-1} \xi^u_c | \xi^u_c)$ and
$$  H_ {\nu} (f_c^{-1} \xi^u_c | \xi^u_c) = \sum_i  \int_{\mathcal{W}^s_{\loc} (\pi(x_i))} \sum_{j} - \nu^u_t (\pi(B_j)) \log (\nu^u_t(\pi(B_j))).$$
Observe that the above sum is over $A^i_c.$

When $h_{v} (f_c) = h_{\mu}  (f,\mathcal{F}^u)$, we must have equality in the Jensen inequality.
Hence, we shown that $\pi_*(\mu^u_{x})=\nu^u_{\pi(x)}$ restricting on the sub algebra generated by $f^{-1}\xi^u$.

Because
$$h_\mu(f^n,\mathcal{F}^u)=nh_\mu(f,\mathcal{F}^u)=nh_{\nu}(f^c) = h_{\nu} (f_ c^n),$$
applying a similar argument as above, one can show that
$\pi_*(\mu^u_{x})=\nu^u_{\pi(x)}$ restricting on the sub algebra generated by
$\cB_0=\{f^{-n}\xi^u\}_{n\in \mathbb{N}}$. Observe that $\cB_0$ generates the Boreal $\sigma$-algebra of every $\xi^u(x)$,
hence we show that
$\pi_*(\mu^u_x)=\nu^u_{\pi(x)}$, as claimed.

Now that we have proved Theorem \ref{u-invariantprinciple}, let us show how to conclude the proof of Corollary \ref{exp-invariantprinciple}.

\begin{proof} (of Corollary \ref{exp-invariantprinciple})
By Ledrappier-Young \cite{LY} we have $$h_{\mu}(f) = h (\mu, \mathcal{F}^u).$$

Indeed, the authors in \cite{LY2} define the notion of entropy $h_i$ along $i$-th unstable manifolds $W^i$ for $1 \leq i \leq u.$ Here
$$
W^i(x) = \{y \in M, \limsup_{n \rightarrow \infty} \frac{1}{n} \log d(f^{-n}(x), f^{-n}(y)) \leq - \lambda_i    \}
$$
and $\lambda_1  > \lambda_2 \cdots > \lambda_u$ are the positive Lyapunov exponents of $\mu.$
In particular, in Theorem (C'), they proved that $h_u = h_{\mu}(f).$ See item (iii) after theorem (C') in \cite{LY2}.

Here as the central Lyapunov exponents are non-positive we conclude that $W^u$ coincides with unstable foliation $\mathcal{F}^u.$  Again using \cite{LY2} we see that $h_u = h (\mu, \mathcal{F}^u) $ which yields $h_{\mu}(f) = h (\mu, \mathcal{F}^u).$

As $f_c$ is a factor of $f$ we have that $h_{\nu} (f_c) \leq h_{\mu} (f)$ and this implies $$h_{\mu} (f, \mathcal{F}^u) = h_{\nu} (f_c).$$ Now Theorem \ref{u-invariantprinciple} implies that $\mu$ is $u-$invariant.
\end{proof}

Now, let us give the proof of Corollary \ref{onedimensional}.
\begin{proof} (of Corollary \ref{onedimensional})
If the central foliation is one dimensional then $h_{\mu}(f) = h_{\nu}(f_c).$ Indeed, by Ledrappier-Walter's variational principle \cite{LW},
\begin{equation} \label{LW}
 \sup_{\hat{\mu}: \pi_ * \hat{\mu} = \nu} h_{\hat{\mu}} (f) = h_{\nu}(f_c) + \int_{M^c} h(f, \pi^{-1} (y)) d\nu(y).
\end{equation}
Since each $\pi^{-1} (y)$ is a circle and its iterates have bounded length we have that $h(f, \pi^{-1} (y)) =0$ that is, fibers does not contribute to the entropy. Hence, by the above equality and the well-known fact that $h_{\mu} (f) \geq h_{\nu} (f_c)$ we conclude that $h_{\mu}(f) = h_{\nu}(f_c)$ and the corollary is immediate from the Theorem \ref{u-invariantprinciple}.

\end{proof}

Finally, let us prove the corollary \ref{invariancelimit}.  By Theorem \ref{u-invariantprinciple}, $\mu_n \in Gibb^u_{\nu}(f_n)$ implies that $h_{\mu_n}(f_n , \mathcal{F}^u_{n}) = h_{\nu}(A)$ where $\mathcal{F}^u_n$ represents the unstable foliation of $f_n.$ By the upper semi-continuity property  (proved in \cite{Y}) we conclude that
$$\lim sup_{n \rightarrow \infty}  h_{\mu_n} (f_n , \mathcal{F}^u_{n})  \leq  h_{\mu} (f , \mathcal{F}^u).
$$
which implies 
$$
h_{\nu}(A)  \leq  h_{\mu} (f , \mathcal{F}^u).
$$
Again using Theorem \ref{u-invariantprinciple}, we have $\mu \in Gibb^u_{\nu}(f).$

\section{Proof of Rigidity of high entropy measures} \label{proofhighentropy}

In this section we prove Theorems \ref{main.nonunihyp} and \ref{main.converging}.
Let us recall some facts about  measures of maximal entropy in our context. As $f_c$ is a transitive Anosov homeomorphism it admits a unique measure of maximal entropy.
From now on, $\nu$ denotes the unique maximal measure of $f_c$ and $\nu^u_{\pi(x)}$ denotes the conditional measure on $W^u_{loc}(\pi(x))$
of $\nu$ corresponding to the measurable partition $\xi_c^u$. Denote by
$$H^s_{\pi(x),\pi(y)}: W^u_{loc}(\pi(x))\to W^u_{loc}(\pi(y)).$$
the holonomy map in each Markov component induced by the stable foliation $\cW^s$. The following result is classical (for instance by means of Margulis construction of measures of maximal entropy): the measure of maximal entropy for $f_c$ has local product structure.
\begin{lemma}
For $\nu$ almost every points $\pi(x), \pi(y) \in A^i_c$ ($i=1,\dots, k$),
$$(H^s_{\pi(x),\pi(y)})_*(\nu^u_{\pi(x)})=\nu^u_{\pi(y)}.$$ A similar statement holds for the disintegration of $\nu$ along stable plaques.
\end{lemma}

In particular, fixing  $ p_i \in A^i_c$, $\nu\mid A^c_i$ can be written as
\begin{equation}\label{eq.productforquotient}
\nu\mid A_c^i= \int_{W^s(p_i)} (H^s_{p_i, q})_*(\nu^u_{p_i})d\nu^{s}(q),
\end{equation}
where $\nu^s$ is the quotient measure on the quotient space $A^i_c/\xi_c^u\cong W^s_{loc}(p_i)$.

\subsection{Some properties of Gibbs measures}
The next proposition is formulated for measures in $\Gibb_{\nu}^u(f)$ such that $\nu$ is the maximal entropy measure of $f_c.$ Although the proposition holds for all invariant measures $\nu$, we formulated it in the case where $\nu$ has local product structure which is sufficient for our purposes.

\begin{proposition}\label{p.Gibbs u state}
Let $f$ be as in Theorem (B) and $\nu$ be the measure of maximal entropy for $f_c$;  then
\begin{itemize}
\item[(a)] $\Gibb_{\nu}^u(f)$ is a compact convex set in the weak-* topology and the extreme points are ergodic;
\item[(b)] For each $\mu\in \Gibb_{\nu}^u(f)$, almost every ergodic component of $\mu$ belongs to $\Gibb_{\nu}^u(f)$.
\end{itemize}
\end{proposition}
\begin{proof}
First consider a coordinate of $A^i_c$
 $$\Phi^i_c: [0,1]^s\times [0,1]^u=\mathcal{I}_c \to A^i_c$$
such that for any $a_c=(a_1,a_2)\in \mathcal{I}_c$ and $x_c=\Phi^i_c(a_c)$:
\begin{itemize}
\item[(i)] $\Phi^i_c(a_1\times [0,1]^u)=W^u_{loc}(x_c)$;
\item[(ii)] $\Phi^i_c([0,1]^s\times a_2)=W^s_{loc}(x_c)$.
\end{itemize}
In these coordinate, the $\Phi^i_c$ image of every horizontal plane is a stable plaque,
and $\Phi^i_c$ image of every vertical plane is a unstable plaque.
Then by \eqref{eq.productforquotient}, the disintegrations of $\nu_i=(\Phi^i_c)^{-1}_*(\nu\mid A^i_c)$ along
the foliation $\{a_1\times [0,1]^u\}_{a_1\in [0,1]^s}$ are all the same, which we denote by $\nu^u_i$;

In the following, we also need a coordinate for $A^i$.
We take each $A^i_c\subset M_c$ with small diameter, such that the central bundle is trivial restricted on $A^i_c$.
Then we can take a continuous coordinate of $A^i$,
$$\Phi^i: [0,1]^s\times [0,1]^u \times S^1=\mathcal{I} \to A^i$$
such that for any $a=(a_1,a_2,a_3)\in \mathcal{I}$ and $x=\Phi^i(a)$:
\begin{itemize}
\item[(i)] $\Phi^i(a_1\times [0,1]^u\times a_3)=\cF^u_{loc}(x)$;
\item[(ii)] $\Phi^i(a_1\times [0,1]^u\times S^1)=\cF^{cu}_{loc}(x)$;
\item[(iii)] $\Phi^i([0,1]^s\times a_2\times S^1)=\cF^{cs}_{loc}(x)$;
\item[(iv)] $\Phi^i(a_1\times a_2 \times S^1)=\cF^c(x)$.
\end{itemize}
Of course, $\mathcal{F}^s$ and $\mathcal{F}^u$ are not necessarily jointly integrable,
If we denote by $\pi_{3}: \mathcal{I}\to [0,1]\times [0,1]$: $\pi_3(a_1\times a_2\times a_3)=a_1\times a_2$,
then it is clear that, we have
\begin{equation}\label{eq.commute}
\Phi^i_c\circ \pi_3=\pi\circ \Phi^i.
\end{equation}

From the definition of $\Gibb^u_\nu$ and \eqref{eq.commute}, the disintegrations of $\mu_i=(\Phi^i)^{-1}_*(\mu\mid A^i)$ along
the foliation
$$\{a_1\times [0,1]^u\times a_3\}_{a_1\in [0,1]^s,a_3\in S^1}$$
equal to $(\Phi^i)^{-1}_*(\mu^u_{\cdot})$, which are all the same and coincide to $\nu^u_i$.

We first prove that $\Gibb^u_\nu(f)$ is compact. Let $\mu_n\in\Gibb^u_\nu$, and $\mu_n\overset{\text{weak}*}{\to} \mu$.
We are going to show that $\mu$ also belongs to $\Gibb^u_\nu$. Then by the coordinate, it is sufficiently to show that, the disintegration of measure $\mu_i$ along the foliation
$$\{a_1\times [0,1]^u\times a_3\}_{a_1\in [0,1]^s,a_3\in S^1}$$
equal to $\nu^u_i$. This is obvious, because each measure $\mu_{n,i}=(\Phi^i)^{-1}(\mu_n\mid A_i)$ can be written as a product
$$d\mu_{n,i}((a_1,a_2,a_3))=d\nu^u_i(a_2) d\tilde{\mu}_{n,i}(a_1\times a_3), $$
where $\tilde{\mu}_{n,i}$ is a probability measure on the space $[0,1]^s\times 0\times S^1$.
Then its limit, $\mu_i$ can be written in the same manner.

Now we are going to prove the second part.
Let $\mu\in \Gibb^u_\nu$ and
write the ergodic decomposition of $\mu$ by
$$\mu= \int_{M / \xi_{erg}} \mu_P d\tilde{\mu}(P),$$ where $\xi_{erg}$ is the measurable partition of $M$ into ergodic components of $\mu.$

Now, we recall the crucial fact   that $\xi^u$ is finer than $\xi_{erg}$. See \cite{LY} (Section 6.2) and  proposition 2.6 in \cite{LS}.
For $x\in M$, denote by $\xi^u (x)$ and $P= \xi_{erg}(x)$ the elements of partitions $\xi^u$ and $\xi_{erg}$
which contain $x$ respectively. Then by the essential uniqueness of the disintegration, for $\mu$ almost
every $x$, the disintegration of $\mu$ along the partition $\xi^u$ on the element $\xi^u (x)$, denoted by $\mu^u_x$,
coincides with the disintegration of the ergodic component $\mu_{P}$ along the partition $\xi^u$.  This  means that, for $\tilde{\mu}$ almost every $P$, the disintegration of $\mu_{P}$ along the partition $\xi^u$
equal to $\mu^u_{x}=\pi^{-1}(\nu^u_{x_c})$ for $\mu_{P}$ almost every $x$, and hence $\mu_{P}$ belongs to $\Gibb^u_\nu$.
\end{proof}

\begin{proposition}\label{p.invariant principle}

 Let $\omega$ be an ergodic maximal measure of $f$ with non-positive (resp. non-negative)
center exponent then $\pi_*(\omega)=\nu$ and $\omega\in \Gibb^u_{\nu}(f)$ (resp. $\omega\in \Gibb^s_\nu(f)$). Moreover, if
$\omega\in \Gibb^u_\nu(f)$ , then $\mu$ is $u-$invariant.

\end{proposition}

\begin{proof}
To prove the first part of the proposition take $\omega$ ergodic measure of maximal entropy with non-positive center exponent. By Ledrappier-Walter's variational principle and one dimensionality of central foliation, $\pi_*(\omega)$ is the measure of maximal entropy for $f_c$, that is $\pi_*(\omega) = \nu.$

By Rokhlin disintegration theorem, there is a system of conditional probability measures along center foliation. By invariance principle (Corollary \ref{exp-invariantprinciple}) $\{\omega_x\}$ is $u-$invariant which is the same to say $\omega \in Gibb^u_{\nu}$ by Proposition \ref{gibbs-uinvariant}.

The second part of the proposition is immediate from Proposition \ref{gibbs-uinvariant}.

\end{proof}

\begin{corollary}
If $\omega \in \Gibb^u_\nu(f) \cap \Gibb^s_\nu(f)$ then  $f$ is of rotation type, and
there is a family of conditional measures
$\omega^c$ along the center foliation, such that $\mu^c$ varies continuously respect to the
center leaves, and is invariant under stable, and unstable holonomies $H^{*}, *=s,u$.
\end{corollary}

\begin{proof}
By above proposition $\omega$ is both $u$ and $s-$state. As the quotient measure $\pi_{*}(\mu) = \nu$ has local product structure the corollary is immediate from invariance principle (see \cite{AV}).
\end{proof}

We need a main property on the partial entropy along expanding foliations which is the following  upper semi-continuity result:
\begin{proposition}[\cite{Y}]\label{p.uppersemicontinuous}

Let $\cF$ be an expanding foliation of $f$, and  $\mu_n$ be a sequence of invariant probability of $f$.
Suppose $\mu_n$ converge to $\mu_0$ in the weak-* topology, then
$$\limsup_{n \rightarrow \infty} h_{\mu_n}(f,\cF)\leq  h_{\mu_0}(f,\cF).$$
\end{proposition}

In the following, we show the idea of the proof of the above proposition when $f\in \SPH_1$, since
in this case the discussion is much simpler.

\begin{proof}[Sketch of proof:]
We need to show
\begin{equation}\label{eq.semicontinuous}
\limsup_{n \rightarrow \infty} H_{\mu_n}(f^{-1}\xi^u\mid \xi^u)\leq H_\mu(f^{-1}\xi^u\mid \xi^u).
\end{equation}

Fix a point $x_i\in \pi^{-1}(p_i)$ for each $i=1,\dots, k$. Consider a sequence of finite partitions
$\cC_{i,1}\leq\cC_{i,2}\leq \dots$ on
$$\cF^{cs}_{loc}(x_i)=\pi^{-1}(W^s_{loc}(p_i)),$$
such that
\begin{itemize}
\item[(A)] $\diam(\cC_{i,t})\to 0$;

\item[(B)] For any $i,t$ and any element $C$ of $\cC_{i,t}$, $\mu_n(\cup_{x\in \partial C}\xi^u(x))=0$ for every $n$.
\end{itemize}
Then for every $t>0$, there are two finite partitions $\tilde{\cC}_{t}$ and $\overline{\cC}_t$:

$$\tilde{\cC}_{t}=\{\cup_{x\in C} \xi^u(x); C \text{ is an element of $\cC_{i,t}$ for some $1\leq i \leq k$}\};$$
and
$$
\overline{\cC}_t=\{\tilde{C}\cap \pi^{-1}(P), \text{ where $\tilde{C}$ is an element of $\tilde{\cC}_t$ and $P$ is an element of $f^{-1}(\xi^u_c)$}\}.
$$

Then $\tilde{\cC}_t\leq \overline{\cC}_t$ and both sequence of partitions are increasing. Moreover, we have
\begin{itemize}
\item[(i)] $\tilde{\cC}_t \nearrow \xi^u$;
\item[(ii)] $\overline{\cC}_t\nearrow f^{-1}\xi^u$;
\item[(iii)] $\mu_n(\partial(\tilde{\cC}_t))=0$ and $\mu_n(\partial(\overline{\cC}_t))=0$ for every $t,n$.
\end{itemize}
We claim that for each $n\in \mathbb{N}$,
\begin{equation}\label{eq.finteapproach}
H_{\mu_n}(\overline{\cC}_t\mid \tilde{\cC}_t)\searrow H_{\mu_n}(f^{-1}(\xi^u)\mid\xi^u).
\end{equation}
To prove this claim, first by (i):
$$H_{\mu_n}(\overline{\cC}_1 \mid \tilde{\cC}_t)\searrow H_{\mu_n}(\overline{\cC}_1\mid\xi^u).$$
It follows that
$$H_{\mu_n}(\overline{\cC}_1\vee\tilde{\cC}_t  \mid \tilde{\cC}_t)\searrow H_{\mu_n}(\overline{\cC}_1\vee \xi^u\mid\xi^u).$$
From the construction of the partitions $\tilde{\cC}_t$ and $\overline{\cC}_t$, it is easy to see that
$\overline{\cC}_1\vee \tilde{\cC}_t=\overline{\cC}_t$ and $\overline{\cC}_1\vee \xi^u=f^{-1}(\xi^u)$, and the
proof of this claim follows  immediately.

Now we continue the proof of \eqref{eq.semicontinuous}. By the above claim, for any $\vep>0$, there is $T$
sufficiently large, such that $H_{\mu_0}(\overline{\cC}_T\mid \tilde{C}_T)<H_{\mu_0}(f^{-1}(\xi^u)\mid \xi^u)+\vep$.
Because both partitions $\overline{\cC}_T$ and $\tilde{\cC}_T$ are finite and their boundaries have zero measure
for any $\mu_n$, we have that
\begin{equation*}
\begin{aligned}
H_{\mu_0}(\overline{\cC}_T\mid \tilde{\cC}_T)&=H_{\mu_0}(\overline{\cC}_T)-H_{\mu_0}(\tilde{\cC}_T)\\
&=\lim H_{\mu_n}(\overline{\cC}_T)-\lim H_{\mu_n}(\tilde{\cC}_T)\\
&=\lim H_{\mu_n}(\overline{\cC}_T\mid \tilde{\cC}_T).\\
\end{aligned}
\end{equation*}
Again by the above claim, we show that
\begin{equation*}
\begin{aligned}
\limsup H_{\mu_n}(f^{-1}(\xi^u)\mid \xi^u)&\leq \lim H_{\mu_n}(\overline{\cC}_T\mid \tilde{\cC}_T)\\
&=H_{\mu_0}(\overline{\cC}_T\mid \tilde{\cC}_T)\\
&\leq H_{\mu_0}(f^{-1}(\xi^u)\mid \xi^u)+\vep.\\
\end{aligned}
\end{equation*}
Since we can take $\vep$ arbitrarily small, we finished the proof of \eqref{eq.semicontinuous}.

\end{proof}

\subsection{Proof of Theorems \ref{main.nonunihyp} and \ref{main.converging}}
Let us first prove Theorem ~\ref{main.converging}. As the center exponents of $\mu_n$ are non-positive, the Pesin unstable lamination coincides
to the unstable foliation. Then we have
$$h_{\mu_n}(f)=h_{\mu_n}(f,\cF^u),$$
which was proved in Ledrappier-Young \cite[Corollary 5.3]{LY} under the assumption that $f$ is $C^2$. See the proof of Corollary \ref{exp-invariantprinciple} for more details.

By our assumption and Proposition~\ref{p.uppersemicontinuous},
$$h_\mu(f,\cF^u)\geq \limsup_{n\to \infty}h_{\mu_n}(f,\cF^u).$$
But by variational principle, it is clear that
$h_\mu(f,\cF^u)\leq h_\mu(f)\leq  h_{top}(f)$.
Hence, we have the equality: $h_\mu(f,\cF^u)=h_{top}(f) = h_{\nu} (f_c)$. Then as a corollary of
Theorem \ref{u-invariantprinciple}, $\mu\in \Gibb^u_\nu$.

Now it is sufficient to prove the following lemma,
\begin{lemma} If $f$ is not of rotation type, then $\Gibb^u_\nu(f) = V^{-}$ where $V^{-}$ the set of invariant probabilities which are combination of $\mu^-_1,\dots, \mu^-_{k(-)}$.
\end{lemma}

\begin{proof}
By Proposition~\ref{p.invariant principle}, $\mu^-_i\in \Gibb^u_\nu(f)$ for each $1\leq i \leq k(-)$ and consequently $$V^-\subset \Gibb^u_\nu(f).$$
Now let $\mu$ belongs to $\Gibb^u_\nu(f)$, by Proposition~\ref{p.Gibbs u state} (b), we may assume that it is ergodic.
As $\mu$ projects to $\nu$ which a maximal measure for $f_c$ we have that  $\mu$ is a maximal measure, then, by Theorem~\ref{dichotomy},
the center exponent of $\mu$, $\lambda^c(\mu)$, can not vanish. We claim that $\lambda^c(\mu)<0$.
Suppose by contradiction that $\lambda^c(\mu)>0$, then
by Proposition~\ref{p.Gibbs u state}, $\mu\in \Gibb^s_\nu(f)$, which contradicts Proposition~\ref{p.invariant principle}.

Therefore $\mu$ is a maximal measure with negative center exponent, apply Theorem~\ref{dichotomy} again,
$\mu=\mu^-_i$ for some $1\leq i \leq k(-)$. Thus $$\Gibb^u_\nu(f)\subset V^-$$ and the proof is complete.

\end{proof}
So we have proved Theorem \ref{main.converging} and the proof of Theorem \ref{main.nonunihyp} is a simple corollary. Indeed, suppose that by contradiction there is a sequence of ergodic measures $\mu_n$ such that $h_{\mu_n} \rightarrow h_{top}(f)$ and  without loss of generality we assume that $\lambda^c(\mu)_n \leq 0$ and converge to zero. Let $\mu$ be an accumulation point of $\mu_n.$
By continuity argument $$\lambda^c(\mu_n) \rightarrow \lambda^c(\mu):= \int_{M} \log Df|_{E^c(x)} d \mu(x).$$
By Theorem ~\ref{main.converging}, $\mu$ is a convex combination of $\mu^{-}_1, \cdots, \mu^{-}_{k^{-}}$ so, $$|\lambda^c (\mu)|\geq \min \{ \lambda^c(\mu^ {-}_1), \cdots, \lambda^c ( \mu^{-}_{k^{-}})   \}, $$
which is a contradiction.

\end{document}